\documentclass[12pt]{article}
\usepackage[notref,notcite]{showkeys}
\usepackage{graphicx}
\usepackage{amsmath,amsthm,amssymb,enumerate}
\usepackage{euscript,mathrsfs}
\usepackage[left=2cm,right=2cm,top=3.5cm,bottom=3.5cm]{geometry}
\usepackage{color}

\usepackage{soul}

\catcode`\@=11 \@addtoreset{equation}{section}

\catcode`\@=12

\allowdisplaybreaks

\newtheorem{Theorem}{Theorem}[section]
\newtheorem{Proposition}[Theorem]{Proposition}
\newtheorem{Lemma}[Theorem]{Lemma}
\newtheorem{Corollary}[Theorem]{Corollary}

\theoremstyle{definition}
\newtheorem{Definition}[Theorem]{Definition}

\newtheorem{Remark}[Theorem]{Remark}

\newcommand{\bTheorem}[1]{
\begin{Theorem} \label{T#1} }
\newcommand{\eT}{\end{Theorem}}

\newcommand{\bProposition}[1]{
\begin{Proposition} \label{P#1}}
\newcommand{\eP}{\end{Proposition}}

\newcommand{\bLemma}[1]{
\begin{Lemma} \label{L#1} }
\newcommand{\eL}{\end{Lemma}}

\newcommand{\bCorollary}[1]{
\begin{Corollary} \label{C#1} }
\newcommand{\eC}{\end{Corollary}}

\newcommand{\bRemark}[1]{
\begin{Remark} \label{R#1} }
\newcommand{\eR}{\end{Remark}}

\newcommand{\dy}{{\rm d}y}
\newcommand{\bDefinition}[1]{
\begin{Definition} \label{D#1} }
\newcommand{\eD}{\end{Definition}}

\newcommand{\Q}{\mathbb{T}^d}

\newcommand{\K}{\mathcal{K}}

\newcommand{\Pim}{\Pi_{\mathcal{T}}}

\newcommand{\tvU}{\widetilde{\vU}}

\newcommand{\jump}[1]{\left[ \left[ #1 \right] \right]}

\newcommand{\vrh}{\vr_h}
\newcommand{\vrhn}{\vr_{h_n}}
\newcommand{\vmh}{\vm_h}
\newcommand{\vmhn}{\vm_{h_n}}

\newcommand{\tvm}{\tilde{\vc{m}}}
\newcommand{\bfphi}{\boldsymbol{\varphi}}

\newcommand{\Eh}{E_h}

\newcommand{\ds}{\,\mathrm{d}S(x)}

\newcommand{\bFormula}[1]{
\begin{equation} \label{#1}}
\newcommand{\eF}{\end{equation}}

\newcommand{\facesint}{\mathcal{E}}
\newcommand{\grid}{\mathcal{T}}
\newcommand{\facesK}{\mathcal{E}(K)}
\newcommand{\vuh}{\vu_h}

\newcommand{\intSh}[1] {\int_{\sigma} #1 \ds }
\newcommand{\TS}{\Delta t}

\newcommand{\co}[2]{{\rm co}\{ #1 , #2 \}}
\newcommand{\Ov}[1]{\overline{#1}}

\newcommand{\aleq}{\stackrel{<}{\sim}}

\newcommand{\toNY}{\xrightarrow{\mathcal{N}(Y)}}
\newcommand{\toKY}{\xrightarrow{\mathcal{K}(Y)}}
\newcommand{\toWK}{\xrightarrow{weak-\mathcal{K}}}
\newcommand{\toSK}{\xrightarrow{strong-\mathcal{K}}}
\newcommand{\toN}{\xrightarrow{\mathcal{N}}}

\newcommand{\vr}{\varrho}

\newcommand{\tvr}{\tilde \vr}

\newcommand{\vu}{\vc{u}}
\newcommand{\vm}{\vc{m}}

\newcommand{\vn}{\vc{n}}

\newcommand{\vc}[1]{{\bf #1}}

\newcommand{\Div}{{\rm div}_x}
\newcommand{\Grad}{\nabla_x}

\newcommand{\dx}{\,{\rm d} {x}}

\newcommand{\dt}{\,{\rm d} t }

\newcommand{\vU}{\vc{U}}

\newcommand{\intO}[1]{\int_{\Q} #1 \ \dx}

\newcommand{\vv}{\vc{v}}

\newcommand{\D}{{\rm d}}
\newcommand{\ep}{\varepsilon}

\def\softd{{\leavevmode\setbox1=\hbox{d}%
          \hbox to 1.05\wd1{d\kern-0.4ex{\char039}\hss}}}
\definecolor{Cgrey}{rgb}{0.85,0.85,0.85}
\definecolor{Cblue}{rgb}{0.50,0.85,0.85}
\definecolor{Cred}{rgb}{1,0,0}
\definecolor{fancy}{rgb}{0.10,0.85,0.10}

\newcommand\Cbox[2]{%
    \newbox\contentbox%
    \newbox\bkgdbox%
    \setbox\contentbox\hbox to \hsize{%
        \vtop{
            \kern\columnsep
            \hbox to \hsize{%
                \kern\columnsep%
                \advance\hsize by -2\columnsep%
                \setlength{\textwidth}{\hsize}%
                \vbox{
                    \parskip=\baselineskip
                    \parindent=0bp
                    #2
                }%
                \kern\columnsep%
            }%
            \kern\columnsep%
        }%
    }%
    \setbox\bkgdbox\vbox{
        \color{#1}
        \hrule width  \wd\contentbox %
               height \ht\contentbox %
               depth  \dp\contentbox
        \color{black}
    }%
    \wd\bkgdbox=0bp%
    \vbox{\hbox to \hsize{\box\bkgdbox\box\contentbox}}%
    \vskip\baselineskip%
}

\date{}

\begin{document}


\title{$\K-$convergence as a new tool in  numerical analysis}

\author{Eduard Feireisl\thanks{The research of E.F. and H.M.~leading to these results has received funding from the
Czech Sciences Foundation (GA\v CR), Grant Agreement
18--05974S. The Institute of Mathematics of the Academy of Sciences of
the Czech Republic is supported by RVO:67985840.\newline
\hspace*{1em} $^\spadesuit$M.L. has been funded by the Deutsche Forschungsgemeinschaft (DFG, German Research Foundation) - Project number 233630050 - TRR 146 as well as by  TRR 165 Waves to Weather.} $^{, \clubsuit}$
\and M\' aria Luk\' a\v cov\' a -- Medvi\softd ov\' a$^{\spadesuit}$ \and
Hana Mizerov\' a$^{*, \dagger}$
}

\date{\today}

\maketitle

\bigskip

\centerline{$^*$ Institute of Mathematics of the Academy of Sciences of the Czech Republic}
\centerline{\v Zitn\' a 25, CZ-115 67 Praha 1, Czech Republic}
\centerline{feireisl@math.cas.cz, mizerova@math.cas.cz}

\bigskip

\centerline{$^\clubsuit$ Institute of Mathematics, TU Berlin}
\centerline{Strasse des 17. Juni, Berlin, Germany}

\bigskip
\centerline{$^\spadesuit$ Institute of Mathematics, Johannes Gutenberg-University Mainz}
\centerline{Staudingerweg 9, 55 128 Mainz, Germany}
\centerline{lukacova@uni-mainz.de}

\bigskip
\centerline{$^\dagger$ Department of Mathematical Analysis and Numerical Mathematics}
\centerline{Faculty of Mathematics, Physics and Informatics of the Comenius University}
\centerline{Mlynsk\' a dolina, 842 48 Bratislava, Slovakia}

\begin{abstract}

We adapt the concept of $\K-$convergence of Young measures to the sequences of
approximate solutions resulting from numerical schemes. We obtain new results on pointwise convergence of numerical
solutions in the case when solutions of the limit continuous problem possess minimal regularity. We apply the abstract theory
to a  finite volume method for the isentropic Euler system describing the motion
of a compressible inviscid fluid. The result can be seen as a nonlinear version of the fundamental Lax equivalence theorem.

\end{abstract}


\tableofcontents

\section{Introduction}
\label{I}

A celebrated deep result of Koml\' os \cite{Kom} asserts that any sequence $\{ F_n \}_{n=1}^\infty$
of uniformly $L^1-$bounded real valued functions on a set $Q \subset R^K$ admits a subsequence $\{ F_{n_k} \}_{k=1}^\infty$ such that the arithmetic averages
$\frac{1}{N} \sum_{k=1}^N  F_{n_k}$ converge a.a. to a function $F \in L^1(Q)$.
Moreover, any subsequence of $\{ F_{n_k} \}_{k=1}^\infty$ enjoys the same property. The result has been adapted by Balder \cite{Bald} who introduced the concept
of $\K$(Koml\' os)-convergence for sequences of Young measures.
The use of Young measures in analysis of nonlinear PDEs is not new, see, e.g., DiPerna \cite{Diperna83, Diperna85} or Tartar \cite{t1,t2}.
More recently, the Young measures have been used as an efficient tool in the analysis of certain numerical schemes, see \cite{FeiLuk}, \cite{FLM}, \cite{FeiLukMizShe} or, from a rather different point of view, Fjordholm et al. \cite{FjKaMiTa}, \cite{FjMiTa1}. Our goal in the present paper is to extend the notion of $\K-$convergence
to families of numerical solutions to problems arising in continuum fluid dynamics. As an iconic example, we have chosen the isentropic Euler system describing the time evolution of a compressible  inviscid fluid.

\subsection{Young measures and numerical analysis of essentially ill--posed problems}
\label{YN}

The equations describing the motion of inviscid fluids in  continuum mechanics give rise to mathematically ill--posed
problems. There is a large number of convincing examples that the Euler system
admits infinitely many physically admissible (weak) solutions, in particular in the natural multidimensional setting,
for a generic class of initial data,
see Chiodaroli et al. ~\cite{Chiod, ChiDelKre, ChiKre, ChKrMaSwI}, among others.
In the light of these results, the question of \emph{convergence} of the approximate schemes used in the numerical analysis of the Euler system becomes of fundamental importance. It is our goal in the present paper to address these issues in a new framework
based on the theory of measure--valued solutions. In particular, we adapt the concept of $\K-$convergence, developed in the context of Young measures by Balder \cite{Bald}, to show:
\begin{itemize}
\item \emph{pointwise} convergence of arithmetic averages (Cesaro means) of numerical solutions to a generalized (dissipative) solution of the limit system, even if the
latter may admit oscillatory (wild) solutions;
\item criteria for \emph{unconditional} convergence in the case the limit system is not uniquely solvable.
\end{itemize}

We develop a general framework based on the theory of Young measures to study problems of convergence of numerical methods.
To illustrate the implications of abstract results, we apply the theory to a finite volume method for the isentropic Euler system.
In particular, we show that the arithmetic averages of numerical solutions converge pointwise to a generalized \emph{dissipative} solution of the Euler system introduced in \cite{BreFeiHof19}. We also show how the presence of
oscillations in families of numerical solutions may provide an evidence that the limit problem exhibits singularities.

 Our approach bears some similarities with the recent works of Fjordholm et al.~\cite{FjKaMiTa, FjMiTa1}, who studied the convergence of entropy stable finite volume schemes to a measure--valued solution of the Euler equations.  The main difference lies in the way of averaging procedure. While Fjordholm et al.~average over different solutions of the initially perturbed problems and investigate only the convergence of statistical modes, we introduce a new concept of
$\K-$convergence yielding the averaging over numerical solutions for different mesh steps without any perturbation of initial data.

\subsection{General strategy}

We start by recalling the necessary preliminary material from the theory of Young measures in Section \ref{P}. In particular, we introduce the concept of narrow and $\K-$convergence for Young measures.

In Section \ref{A}, we set up a general framework for studying convergence of numerical schemes. We extend the concept of $\K-$convergence to sequences of numerical solutions
by identifying them with sequences of Young measures.
We distinguish between \emph{weak$-\K$} or \emph{strong$-\K$ convergence} reflecting presence or absence of oscillations, respectively.
In particular, we recover a variant of Koml\' os' result applied to sequences of numerical solutions.

In Section \ref{E}, we collect the necessary theoretical results concerning the isentropic Euler system.
Following \cite{BreFeiHof19} we introduce the concept of dissipative solution that can be seen as a barycenter of the Young measure associated to a suitable
measure--valued solution of the limit system. We also state the relevant weak--strong uniqueness principle.

Finally, in Section \ref{N}, we apply the abstract theory to
a finite volume scheme for the isentropic Euler system. We show that (up to a subsequence) numerical solutions $\K-$converge to a dissipative solution in the asymptotic limit of vanishing discretization step, see Theorem \ref{NT1}. We also identify a large class of weak solutions to the
Euler system for which the scheme converges unconditionally and pointwise, see Theorem \ref{NT2}. Finally, we clarify how the presence of oscillations in the
sequence of numerical solutions may indicate singularities for the limit system, see Theorem \ref{NT4}.

\section{Preliminaries from the theory of Young measures}
\label{P}

We recall some basic facts from the theory of Young measures, in particular the concept of narrow and $\K-$convergence. The reader
may consult the monographs by Balder \cite{Bald} or Pedregal \cite{PED1} for details.

\subsection{Physical space, phase space}

We consider problems defined on the \emph{physical space} $Q \subset R^K$. In the case of evolutionary differential equations, the \emph{physical
space} coincides with the space--time cylinder,
\[
Q = \left\{ (t,x)\ \Big| \ t \in (0,T), \ x \in \Omega \subset R^d \right\},\ d = 1,2,3,\ K = d+1
\]
where $(0,T)$ is the relevant time interval and $\Omega$ the spatial domain.

The \emph{phase space} $S \subset R^M$ characterizes the state of the modeled system at  any given time and spatial position.
For instance the density $\vr$ and the velocity $\vu$ in the models of compressible fluids range in the phase space
\[
{S} = \left\{ [\vr, \vu]\ \Big| \  \vr \in [0, \infty),\ \vu \in R^d \right\} \subset R^{d+1},\ d=1,2,3, \ M = d+1.
\]

If not stated otherwise, the physical space $Q$ will be a \emph{bounded} domain in $R^K$ and its elements denoted by the symbol $y \in Q$;
the phase space $ S$ will be a Borel subset of $R^M$, with elements $\vU \in S$. Most of the applications presented below can be extended to a general (unbounded) open set in $Q \in R^K$.

\subsection{Probability measures}

Let $\mathcal{P} (S)$ denote the space of all regular Borel probability measures $\mathcal{V}$ on $R^M$ such that
\[
\mathcal{V} (S) = 1.
\]
In view of the Riesz representation theorem, any $\mathcal{V}$ can be identified with a non--negative linear functional
on the space $C_c(R^M)$, we denote
\[
\left< \mathcal{V} ; g \right> \equiv
\left< \mathcal{V} ; g(\tvU) \right> = \int_S g (\vU) \ \D \mathcal{V}(\vU) \ \mbox{for any}\ g \in C_c(R^M).
\]
If $\mathcal{V}$ has finite first moment, we denote
\[
\left< \mathcal{V} ; \tvU \right> \in S
\
\mbox{-- the \emph{barycenter} of}\ \mathcal{V} \ \mbox{on}\ S.
\]
\begin{Definition}[Narrow convergence] \label{PD1}

Let $\{ \mathcal{V}_n \}_{n=1}^\infty$ be a sequence of probability measures in $\mathcal{P}(S)$. We say that
$\{ \mathcal{V}_n \}_{n=1}^\infty$ \emph{narrowly converges} to a measure $\mathcal{V}$,
\[
\mathcal{V}_n \toN \mathcal{V} \in \mathcal{P}(S),
\]
if
\[
\left< \mathcal{V}_n; g \right> \to \left< \mathcal{V}; g \right> \ \mbox{for any}\ g \in C_c(R^M).
\]

\end{Definition}

\begin{Remark} \label{PR1}

Narrow convergence is weak-(*) convergence if we identify $\mathcal{P}(S)$ with a bounded subset of the space of (bounded) Radon measures
$\mathcal{M}_b(R^M)$,
\[
\mathcal{M}_b(R^M) = [C_0(R^M)]^*.
\]
A necessary and sufficient condition for the limit to be
a probability measure is \emph{uniform tightness} of the sequence $\{ \mathcal{V}_n \}_{n=1}^\infty$:
For any $\ep > 0$, there exists a compact set $K \subset R^M$ and $m = m(\ep, K)$ such that
\[
\mathcal{V}_n (K) > 1 - \ep \ \mbox{for all}\ n \geq m(\ep, K).
\]

\end{Remark}

\subsection{Young measures}

A Young measure $\mathcal{V}_Q = \left\{ \mathcal{V}_y \right\}_{y \in Q}$ is a family of probability measures
$\mathcal{V}_y \in \mathcal{P}(S)$ parametrized by $y \in Q$. More specifically,

\begin{Definition}[Young measure] \label{PD2}

A \emph{Young measure} $\mathcal{V}_Q$ on the set $Q$ is a mapping
\[
\mathcal{V}_Q = \{ \mathcal{V}_y \}_{y \in Q},\
\mathcal{V}_y : y \in Q \mapsto \mathcal{V}_y \in \mathcal{P}(S)
\]
that is  weakly-(*) measurable,
\[
\mathcal{V}_Q \in L^\infty_{{\rm weak-(*)}}(Q; \mathcal{M}_b(R^M)),
\]
meaning the function
\[
y \in Q \mapsto \left< \mathcal{V}_y; g \right>
\]
is measurable for any $g \in C_c(R^M)$.

\end{Definition}

\begin{Definition}[Narrow convergence of Young measures] \label{PD3}

Let $\{ \mathcal{V}^n_Q  \}_{n = 1}^\infty$ be a sequence of Young measures on $Q$. We say that
$\{ \mathcal{V}^n_Q \}_{n = 1}^\infty$ \emph{narrowly convergences} to
a Young measure $\mathcal{V}_Q $,
\[
\mathcal{V}^n_Q \toNY  \mathcal{V}_Q \ \mbox{as}\ n \to \infty,
\]
if
\[
\int_Q \left< \mathcal{V}^n_y, g(y, \cdot ) \right>  \dy
\to \int_Q \left< \mathcal{V}_y, g(y, \cdot ) \right>  \dy
\ \mbox{for any}\ g \in L^1(Q; C_0(R^M)).
\]

\end{Definition}

\begin{Remark} \label{PR2}

Narrow convergence is weak-(*) convergence if we identify the space $L^{\infty}(Q; \mathcal{M}_b(R^M))$ with the dual
to $[L^1(Q; C_0(R^M))]^*$. The limit is again a Young measure if the averages
\[
\frac{1}{|Q|}  \int_Q \mathcal{V}^n_y \dy \in \mathcal{P}(S)
\]
are uniformly tight in $\mathcal{P}(S)$. If $|Q| = \infty$ we require the above relation to hold for any compact $\tilde{Q}
\subset Q$, $|\tilde{Q}| < \infty$.

\end{Remark}

\subsubsection{$\K-$convergence}

Following Balder \cite{Bald}, we introduce the $\K-$convergence of Young measures.

\begin{Definition} [$\K-$convergence of Young measures] \label{PD4}

Let $\{ \mathcal{V}^n_Q \}_{n = 1}^\infty$ be a sequence of Young measures on $Q$. We say that
$\{ \mathcal{V}^n_Q \}_{n = 1}^\infty$ \emph{$\K-$convergences} to
a Young measure $\mathcal{V}_Q$,
\[
\mathcal{V}_Q^n \toKY  \mathcal{V}_Q\ \mbox{as}\ n \to \infty,
\]
if for any subsequence $\left\{ \mathcal{V}^{n_k}_Q \right\}_{k=1}^\infty$,
\[
\frac{1}{N} \sum_{k = 1}^N \mathcal{V}^{n_k}_y \toN \mathcal{V}_y \ \mbox{as}\ N \to \infty\ \mbox{for a.a.}\ y \in Q.
\]

\end{Definition}

We report the following result, see Balder \cite{Bald}.

\begin{Proposition} \label{PP1}
Let $\{ \mathcal{V}^n_Q \}_{n = 1}^\infty$ be a sequence of Young measures on $Q$ such that
\[
\mathcal{V}_Q^n \toKY  \mathcal{V}_Q.
\]

Then
\[
\mathcal{V}_Q^n \toNY  \mathcal{V}_Q.
\]

\end{Proposition}

A fundamental result  is a version of Prokhorov compactness theorem for Young measures,  see Prokhorov \cite{Prokh} and Balder \cite{Bald1}:

\begin{Proposition} \label{PP2}

Let $\{ \mathcal{V}^n_Q \}_{n=1}^\infty$ be a sequence of Young measures such that
the family of probability measures
\begin{equation} \label{P1}
\frac{1}{|Q|} \int_Q \mathcal{V}_y^n \ \dy \in \mathcal{P} (S)
\end{equation}
is tight uniformly in $n=1,2,\dots$.

Then there is a subsequence $n_k \to \infty$ and a Young measure
$\mathcal{V}_Q$
such that
\[
\mathcal{V}^{n_k}_Q \toKY \mathcal{V}_Q \ \mbox{as}\ k \to \infty.
\]

\end{Proposition}

Note that Proposition \ref{PP2} is considerably stronger than the so--called \emph{Fundamental Theorem of the theory of Young measures}
(see Ball \cite{Ball2} or Pedregal \cite{PED1}) that asserts only $\mathcal{N}(Y)$ convergence under the same assumptions.
The present result can be seen as a variant of the celebrated theorem by Koml\' os \cite{Kom} ($\K-$convergence) on a sequence of uniformly bounded real functions.

\begin{Remark} \label{PR3}

The result can be extended on locally compact physical spaces $ Q = \cup_{k=1}^\infty Q_k$, $Q_k$ compact, by requiring
\eqref{P1} on any $Q_k$.

\end{Remark}

\section{$\K-$convergence for sequences of functions}

\label{A}

Our goal is to set up a general framework to extend the concept of $\K-$convergence to sequences of numerical solutions.
To this end we will introduce the concept of $\K-$convergence for sequences of functions.

For $\vU \in S$ we denote $\delta_{\vU} \in \mathcal{P}(S)$ the Dirac mass supported by $\vU$.
Similarly, we can associate to any measurable function $\vU: Q \to S$ a Young measure $\delta_{\vU}$,
\[
(\delta_{\vU})_y = \delta_{\vU(y)} \ \mbox{for a.a.}\ y \in Q.
\]
Conversely, if a Young measure $\mathcal{V}_Q$ has finite first moments a.a. in $Q$, we can  associate to it a measurable function
$\vU:  Q  \to S $ defined through barycenters (expectations),
\begin{equation} \label{NY1}
\vU(y) = \left< \mathcal{V}_y; \tvU \right> \ \in S,\  \mbox{for a.a.}\ y \in Q.
\end{equation}

\subsection{Convergence in averages}

\begin{Definition}[weak$-\K$ property]
\label{D-weak_K}
We say that a sequence of functions $\{\vU_n\}_{n=1}^\infty,$ $\vU_n \in L^1(Q;S),$ has a \emph{weak$-\K$ property} if there exists a subsequence $n_k\to\infty$ such that the associated  Young measures
\begin{align*}
\delta_{\vU_{n_k}} \toKY \mathcal V_Q \quad \mbox{(in the sense of Definition~\ref{PD4}),}
\end{align*}
where $\mathcal V_Q$ is a Young measure on $Q.$
\end{Definition}

As an immediate consequence of Proposition~\ref{PP2} we get:

\begin{Corollary}
Any sequence $\{\vU_n\}_{n=1}^\infty$ satisfying
\begin{align*}
\|\vU_n\|_{L^1(Q;S)} \leq \overline{\vU} \quad \mbox{uniformly for } n=1,\ldots
\end{align*}
has the weak$-\K$ property.
\end{Corollary}
 Here and hereafter $\Ov{\vU}$ denotes a positive constant.

\begin{Definition}[weak$-\K$ convergence]
\label{D-weak_K_conv}
We say that a sequence of functions $\{\vU_n\}_{n=1}^\infty,$ $\vU_n\in L^1(Q;S),$ \emph{weak$-\K$ converges} to $\vU \in L^1(Q;S)$ if the following holds:
\begin{align*}
\delta_{\vU_{n_k}} \toKY \mathcal V_Q \quad \mbox{for a subsequence } n_k \to \infty\quad  \Longrightarrow \quad \vU(y)=\langle\mathcal V_y,\tilde{\vU}\rangle \quad \mbox{for a.a. } y \in Q .
\end{align*}
\end{Definition}

We can rephrase Proposition~\ref{PP2} as follows.
\begin{Proposition}
Any sequence $\{\vU_n\}_{n=1}^\infty$ satisfying
\begin{align*}
\|\vU_n\|_{L^1(Q;S)} \leq \overline{\vU} \quad \mbox{uniformly for } n=1,\ldots
\end{align*}
 admits a subsequence for $n_k \rightarrow\infty$ such that
 \begin{align*}
 \vU_{n_k} \toWK \vU, \qquad \vU \in L^1(Q;S).
 \end{align*}
 In addition,
 \begin{align*}
 \frac 1 N \sum_{k=1}^N \vU_{n_k} \to \vU \quad \mbox{a.e. in } Q.
 \end{align*}
\end{Proposition}

\begin{Proposition}
For any sequence $\{\vU_n\}_{n=1}^\infty,$
\begin{align*}
\vU_n \rightharpoonup \vU \quad \mbox{weakly in } L^1(Q;S)\quad \Longrightarrow  \quad \vU_{n} \toWK \vU.
\end{align*}

\end{Proposition}

\begin{proof}
Suppose that
\begin{align}\label{weakL1}
\vU_n \rightharpoonup \vU \quad \mbox{weakly in } L^1(Q;S)
\end{align}
and $$\delta_{\vU_{n_k}} \toKY \mathcal V_Q \quad \mbox{for a subsequence } n_k \to \infty.$$
We have to show that $\langle \mathcal V_y,\tvU \rangle=\vU(y)$ for a.a. $y \in Q.$ By virtue of Proposition~\ref{PP1} we have $$\delta_{\vU_{n_k}} \toNY \mathcal V_Q.$$
Specifically,
\begin{align}\label{aux_narrow}
\int_Q h(y)g(\vU_{n_k}(y)) \ \dy \longrightarrow \int_Q h(y)\langle \mathcal V_y, g(\tvU)\rangle \ \dy,
\end{align}
for any  $h \in L^\infty(Q)$ and any $g \in C_c(R^M).$ Moreover, in accordance with the De la Vall\' e--Poussin criterion, see Pedregal \cite[Lemma 6.4]{Ped}, and \eqref{weakL1} we have
\begin{align*}
\int_Q \phi(|\vU_n|) \ \dy \leq \overline{\vU}
\end{align*}
for a superlinear function $\phi.$ Consequently, we are allowed to take $g(\vU)=U^i,$ $i=1,2,\ldots,M,$ in \eqref{aux_narrow} which yields  $\langle \mathcal V_y,\tvU \rangle=\vU(y)$ for a.a. $y \in Q.$
\end{proof}

\begin{Definition}[strong$-\K$ convergence]
\label{D-strong_K_conv}
We say that a sequence of functions $\{\vU_n\}_{n=1}^\infty,$ $\vU_n \in L^1(Q;S),$ \emph{strong$-\K$ converges} to $\vU \in L^1(Q;S)$ if the following holds:
\begin{align*}
\delta_{\vU_{n_k}} \toKY \mathcal V_Q \quad \mbox{for a subsequence } n_k \to \infty \quad \Longrightarrow \quad
\mathcal V_y= \delta_{\vU(y)}\quad \mbox{for a.a. } y \in Q.
\end{align*}

\end{Definition}

\begin{Proposition} \label{PPP}
For any sequence $\{\vU_n\}_{n=1}^\infty,$
\begin{align*}
\vU_n \to \vU \quad \mbox{strongly in } L^1(Q;S)\quad \Longrightarrow  \quad \vU_{n} \toSK \vU.
\end{align*}

\end{Proposition}

\begin{proof}
Suppose that
\begin{align}\label{strongL1}
\vU_n \to \vU \quad \mbox{strongly in } L^1(Q;S)
\end{align}
and $$\delta_{\vU_{n_k}} \toKY \mathcal V_Q \quad \mbox{for a subsequence } n_k \to \infty$$
which means that
\begin{align*}
\frac 1 N \sum_{k=1}^N g(\vU_{n_k}(y)) \to \mathcal \langle \mathcal V_y, g \rangle \quad \mbox{a.a. } y \in Q.
\end{align*}
We have to show  $\mathcal V_y= \delta_{\vU(y)},$  meaning  $g(\vU(y))=\langle \mathcal V_y, g \rangle $ for all $g \in C_c(R^M)$ and a.a. $y \in Q.$
To this end it is enough to observe that for any $g \in C_c(R^M)\cap C^1(R^M),$
\begin{align*}
\frac 1 N \sum_{k=1}^N g(\vU_{n_k}) \to g(\vU) \quad \mbox{in } L^1(Q).
\end{align*}
Indeed,  in accordance with \eqref{strongL1},
\begin{align*}
\left|\left| \frac 1 N \sum_{k=1}^N g(\vU_{n_k})- g(\vU)\right|\right|_{L^1(Q)} &\leq \frac 1 N \sum_{k=1}^N \|g(\vU_{n_k}) - g(\vU)\|_{L^1(Q)} \\
&\leq c_g \frac 1  N \sum_{k=1}^N \|\vU_{n_k}-\vU\|_{L^1(Q)} \to 0 \quad \mbox{ as }\  N \to \infty.
\end{align*}
\end{proof}

\begin{Lemma} \label{SL1}
Let $\{\vU_n\}_{n=1}^\infty$ be a sequence of functions in $L^1(Q; S)$ such that
\begin{equation} \label{S12}
\int_Q |\vU_n |\ \dy \leq \Ov{\vU} \ \mbox{for any}\ n=1,\ldots.
\end{equation}
Then the following is equivalent:
\begin{enumerate}
\item
\[
\vU_{n} \to \vU \ \mbox{in measure in}\ Q;
\]

\item
\[
\delta_{\vU_{n}}  \toNY \delta_{\vU } .
\]

\end{enumerate}

In both cases, there is a subsequence  $n_k \to \infty$ such that
\begin{itemize}
\item
\[
\vU_{n_k}(y) \to \vU(y) \ \mbox{as}\ k \to \infty
\]
for a.a. $y \in Q$;
\item
\[
\frac{1}{N} \sum_{k=1}^N g (\vU_{n_k}(y)) \to g (\vU(y)) \ \mbox{as}\ N \to \infty
\]
for any $g \in C_c(R^M)$ and a.a. $y \in Q$.
\end{itemize}

\end{Lemma}

\begin{proof}

{\bf Step 1:}

Suppose that ${\vU_n} \to \vU$ in measure. To show
\[
\delta_{\vU_n}  \toNY  \delta_\vU
\]
it is enough to observe that
\[
\int_Q h(y) g (\vU_n(y))\ \dy  \to \int_Q h(y) g(\vU(y)) \ \dy \ \mbox{for any}\
h \in L^\infty(Q) \ \mbox{and any} \ g \in  C_c(R^M) \cap C^1(R^M).
\]
To see this, for given $\ep > 0$ and $ k> 0$ we write the integral
\[
\begin{split}
\int_Q &h (y) \left[ g (\vU_n(y) ) - g (\vU(y) ) \right] \dy \\
&= \int_{\{ |\vU_n - \vU | \leq k \}}
h(y) \left[ g (\vU_n(y) ) - g (\vU(y) ) \right]  \dy + \int_{\{ |\vU_n - \vU | > k \} }
h (y) \left[ g (\vU_n(y) ) - g (\vU(y) ) \right]  \dy,
\end{split}
\]
where
\[
\left| \int_{\{ |\vU_n - \vU | \leq k \}}
h (y) \left[ g (\vU_n(y) ) - g (\vU(y) ) \right]  \dy \right| \leq k c_1(h,g),
\]
while
\[
\left| \int_{\{ |\vU_n - \vU | > k \} }
h (y) \left[ g (\vU_n(y) ) - g (\vU(y) ) \right]  \dy \right| < \ep c_2(h,g)
\]
if $n$ is large enough. Consequently, choosing $\ep, k$ small and $n$ large we get the desired result.

\medskip

{\bf Step 2:} Assume now that
\[
\int_Q h(y) g (\vU_n(y) )\ \dy  \to \int_Q h (y) g(\vU(y)) \ \dy \ \mbox{for any}\
h \in L^1(Q) \ \mbox{and any} \ g \in C_c (R^M).
\]
This implies, in particular
\[
\int_Q h(y) |g (\vU_n(y) )|^2 \dy  \to \int_Q h (y) |g(\vU(y))|^2 \ \dy \ \mbox{for any}\
h \in L^1(Q) \ \mbox{and any} \ g \in C_c(R^M).
\]
This gives rise to
\[
g(\vU_n) \to g(\vU) \ \mbox{in}\ L^2(Q) \ \mbox{for any}\ g \in C_c(R^M),
\]
yielding
\[
g(\vU_n) \to g(\vU) \ \mbox{in measure in}\ Q
\]
for any $g \in C_c(R^M)$. This, together with hypothesis \eqref{S12}, implies
\[
\vU_n \to \vU \ \mbox{in measure in}\ Q.
\]

{\bf Step 3:}

As the sequence converging in measure contains a subsequence converging pointwise, and according to Proposition~\ref{PP2},
the associated sequence of Young measures contains a $\K-$converging subsequence, the remaining part of the conclusion of
Lemma \ref{SL1} follows.

\end{proof}

\section{Isentropic Euler system}
\label{E}

We consider the isentropic Euler system describing the time evolution of the density $\vr = \vr(t,x)$ and the momentum
$\vm = \vm(t,x)$ of a compressible inviscid fluid:
\begin{equation} \label{E1}
\begin{split}
\partial_t \vr + \Div \vm &= 0 \\
\partial_t \vm + \Div \left( \frac{\vm \otimes \vm}{\vr} \right) + \Grad p(\vr) &= 0,\ p(\vr) = a \vr^\gamma,\ a > 0,
\ \gamma > 1.
\end{split}
\end{equation}
For the sake of simplicity, we consider the space periodic boundary conditions, meaning the spatial domain can be identified with the
flat torus
\begin{equation} \label{E2}
\Q = \left( [-1,1]|_{\{ -1,1 \}} \right)^d, \ d=1,2,3.
\end{equation}
The problem is supplemented by the initial conditions
\begin{equation} \label{E3}
\vr(0, \cdot) = \vr_0, \ \vm(0, \cdot) = \vm_0.
\end{equation}

We consider solutions defined on the time interval $(0,T)$. Accordingly, the relevant physical space is
\[
Q = (0,T) \times \Q \subset R^{d+1}.
\]
As the mass density is {\it a priori} a non--negative quantity, we set the phase space $S$ to be
\[
S = \left\{ [\vr, \vm] \ \Big| \ \vr \in [0, \infty), \ \vm \in R^d \right\}.
\]
Thus $K = M = d+1$, $d=1,2,3$.

We also introduce the \emph{energy inequality} in the form
\begin{equation} \label{E4}
\frac{ \D }{\D t} \intO{  \left[\frac{1}{2} \frac{ |\vm|^2 }{\vr} + P(\vr)   \right]} \leq 0,\
\intO{  \left[\frac{1}{2} \frac{ |\vm|^2 }{\vr} + P(\vr)   \right] (0+, \cdot) } \leq
\intO{ \left[\frac{1}{2} \frac{ |\vm_0|^2 }{\vr_0} + P(\vr_0)   \right] },
\end{equation}
where
\[
P(\vr) = \frac{a}{\gamma - 1} \vr^\gamma
\]
is the \emph{pressure potential}.

As already mentioned in Section \ref{YN}, the Euler system \eqref{E1}--\eqref{E3} is essentially ill--posed on $(0,T)$. Specifically, the unique strong solutions exist only on a possible short time interval $[0, T_{\rm max})$, while the problem admits infinitely many
weak solutions for general initial data. In addition, there are infinitely many weak solutions
for certain initial data even if the energy inequality \eqref{E4} is imposed. On the other hand, the existence of global in time
admissible weak solutions, meaning weak solutions satisfying the energy inequality \eqref{E4}, for \emph{general} initial data is an open problem. To overcome this difficulty, we introduce a class of generalized \emph{dissipative} solutions that exist globally in time for all finite energy initial data.

\subsection{Weak solutions}

We start with the weak formulation of \eqref{E1}--\eqref{E3} that reads:
\begin{equation} \label{E5}
\left[ \intO{ \vr \varphi } \right]_{t = 0}^{t = \tau} =
\int_0^\tau \intO{ \Big[ \vr \partial_t \varphi + \vm \cdot \Grad \varphi \Big] } \dt,\
\vr(0, \cdot) = \vr_0,
\end{equation}
for any $0 < \tau < T$, $\varphi \in C^1([0,T] \times \Q)$;
\begin{equation} \label{E6}
\left[ \intO{ \vm \cdot \bfphi } \right]_{t = 0}^{t = \tau} =
\int_0^\tau \intO{ \Big[ \vm \cdot \partial_t \bfphi + \frac{\vm \otimes \vm}{\vr} : \Grad \bfphi
+ p(\vr) \Div \bfphi \Big] } \dt, \ \vm(0,\cdot) = \vm_0,
\end{equation}
for any $0 < \tau < T$, $\bfphi \in C^1([0,T] \times \Q; R^d)$.

In addition, the total energy is a non--decreasing function of time,
\begin{equation} \label{E7}
\begin{split}
\intO{ \left[ \frac{1}{2} \frac{ |\vm|^2 }{\vr} + P(\vr) \right] (\tau+, \cdot) }
&\leq \intO{ \left[ \frac{1}{2} \frac{ |\vm|^2 }{\vr} + P(\vr) \right] (s-, \cdot) }
\ \mbox{for any}\ \tau \geq s \geq 0,\\ \mbox{where we set}\
\intO{ \left[ \frac{1}{2} \frac{ |\vm|^2 }{\vr} + P(\vr) \right] (0-, \cdot) }
&\equiv \intO{ \left[ \frac{1}{2} \frac{ |\vm_0|^2 }{\vr_0} + P(\vr_0) \right] }.
\end{split}
\end{equation}

\subsection{Dissipative measure--valued solutions}

Following \cite{BreFeiHof19} we introduce a class of generalized solutions using the theory of Young measures.
The leading idea is to replace all nonlinear compositions in \eqref{E5}--\eqref{E7} by the action of a suitable Young measure
$\{ \mathcal{V}_y \}_{y \in Q}$, $\mathcal{V}_y \in \mathcal{P}(S)$. A weak solution would then correspond to
$\mathcal{V} = \delta_{[\vr, \vm]}$. Unfortunately, the only {\it a priori} bounds available result from boundedness of the
total energy uniformly in time. In view of the specific form of the isentropic EOS, this yields
\begin{equation} \label{E7a}
\frac{|\vm|^2}{\vr} \in L^\infty(0,T; L^1(\Q)),\ \vr \in L^\infty(0,T; L^\gamma(\Q)),\
\vm \in L^\infty(0,T; L^{\frac{2 \gamma}{\gamma + 1}}(\Q; R^d)).
\end{equation}
 Our goal is to generate the measure--valued solutions by means of an
energy dissipative numerical scheme.
Given rather poor stability estimates that basically reflect \eqref{E7a},
the energy, the pressure as well as the convective term in the momentum equation \eqref{E6} may develop concentrations that give rise to the so--called concentration
defect measures in the limit equations. We shall therefore make an ansatz
\[
\begin{split}
\vr(t,x) &\approx \left< \mathcal{V}_{t,x}; \tvr \right>,\ \vm(t,x) \approx \left< \mathcal{V}_{t,x}; \tvm \right>, \\
 P(\vr)(t,x) &\approx \left< \mathcal{V}_{t,x}; P(\tvr) \right> +
\mathfrak{C}_{\rm int}(t,x),\ p(\vr)(t,x) \approx \left< \mathcal{V}_{t,x}; p(\tvr) \right> +
(\gamma - 1) \mathfrak{C}_{\rm int}(t,x), \\
\frac{1}{2} \frac{|\vm|^2}{\vr}(t,x) &\approx \frac{1}{2} \left< \mathcal{V}_{t,x}; \frac{|\tvm|^2}{\tvr} \right> +
\mathfrak{C}_{\rm kin}(t,x),
\end{split}
\]
where the energy \emph{concentration defect measures} belong to the class
\[
\mathfrak{C}_{\rm kin},\ \mathfrak{C}_{\rm int} \in L^\infty(0,T; \mathcal{M}^+ (\Q)).
\]

The convective term in the momentum equation is more delicate. We write
\[
\frac{\vm \otimes \vm}{\vr} = 2 \left( \frac{ \vm }{|\vm|} \otimes \frac{ \vm }{|\vm|} \right) \frac{1}{2} \frac{|\vm|^2}{\vr}
\]
seeing that the expression on the right--hand side is a rank-one symmetric matrix with the trace $\frac{|\vm|^2}{\vr}$. This motivates
the following ansatz for the convective term:
\[
\frac{\vm \otimes \vm}{\vr}(t,x) = \left< \mathcal{V}_{t,x}; \frac{\tvm \otimes \tvm}{\tvr} \right>
+ \mathfrak{C}_{\rm conv}(t,x),
\]
where
\begin{equation} \label{E8}
\begin{split}
\mathfrak{C}_{\rm conv} &\in L^\infty(0,T; \mathcal{M}(\Q; R^{d \times d}_{{\rm sym}})),\\
\int_{\Q} \mathbb{M} : \D \mathfrak{C}_{\rm conv}(t) &\geq 0 \ \mbox{for any}\ \mathbb{M} \in C(\Q; R^{d \times d}_{\rm sym}),\
\mathbb{M} \geq 0, \ \mbox{and a.a.}\ t \in (0,T), \\
\frac{1}{2} \int_{\Q} h \mathbb{I} : \D \mathfrak{C}_{\rm conv}(t) &=  \int_{\Q} h \ \D \mathfrak{C}_{\rm kin}(t)
\ \mbox{for any}\ h \in C(\Q), \ \mbox{and a.a.}\ t \in (0,T).
\end{split}
\end{equation}

\begin{Remark} \label{RRP1}

The second condition in \eqref{E8} indicates that $\mathfrak{C}(t, \cdot)$ is positively define, while the third
one can be interpreted as
\[
\frac{1}{2} {\rm trace} [\mathfrak{C}_{\rm conv}] = \mathfrak{C}_{\rm int}.
\]

\end{Remark}

\begin{Definition}[Dissipative measure--valued solution] \label{ED1}

We say that a Young measure \\ $\mathcal{V} \in L^{\infty}_{{\rm weak-(*)}}((0,T) \times \Q; \mathcal{P}(S))$, and the concentration defect measures
$\mathfrak{C}_{{\rm kin}} \in L^\infty(0,T; \mathcal{M}^+(\Q))$, $\mathfrak{C}_{{\rm int}} \in L^\infty(0,T; \mathcal{M}^+(\Q))$,
$\mathfrak{C}_{{\rm conv}} \in L^\infty(0,T; \mathcal{M}(\Q; R^{d \times d}_{\rm sym} ))$, satisfying the compatibility condition
\eqref{E8}, are \emph{dissipative measure valued solution} of the Euler system \eqref{E1}, \eqref{E2} with the initial data
\[
[\vr_0 , \ \vm_0, \ E_0], \  \vr_0 \geq 0, \ \intO{\frac 1 2 \frac{|\vm_0|^2}{\vr_0} + P(\vr_0) } \leq E_0
\]
if:

$\bullet$
\begin{equation} \label{E9}
\intO{ \left< \mathcal{V}_{t, x}; \tvr \right>}  - \intO{ \vr_0 \varphi } =
\int_0^\tau \intO{ \Big[ \left< \mathcal{V}_{t, x}; \tvr \right> \partial_t \varphi + \left< \mathcal{V}_{t, x}; \tvm \right> \cdot \Grad \varphi \Big] } \dt
\end{equation}
for any $0 < \tau < T$, $\varphi \in C^1([0,T] \times \Q)$;

$\bullet$
\begin{equation} \label{E10}
\begin{split}
\intO{ \left< \mathcal{V}_{t,x}; \tvm \right> \cdot \bfphi } &- \intO{ \vm_0 \cdot \bfphi }\\
 &=
\int_0^\tau \intO{ \Big[ \left< \mathcal{V}_{t,x}; \tvm \right> \cdot \partial_t \bfphi +
\left< \mathcal{V}_{t,x}; \frac{\tvm \otimes \tvm}{\tvr} \right> : \Grad \bfphi
+ \left< \mathcal{V}_{t,x};p(\tvr) \right> \Div \bfphi \Big] } \dt \\
&+ \int_0^\tau \int_{\Q} \Grad \bfphi : \D \mathfrak{C}_{\rm conv}(t, \cdot) \dt
+ (\gamma - 1) \int_0^\tau \int_{\Q} \Div \bfphi \ \D \mathfrak{C}_{\rm int}(t, \cdot) \dt
\end{split}
\end{equation}
for any $0 < \tau < T$, $\bfphi \in C^1([0,T] \times \Q; R^d)$;

$\bullet$
the total energy $E$ is a non--increasing function on $[0,T]$,
\begin{equation} \label{E11}
E(\tau) = \intO{ \left< \mathcal{V}_{\tau,x}; \frac{1}{2} \frac{|\tvm|^2}{\tvr} + P(\tvr) \right> } +
\int_{\Q} \Big[ \D \mathfrak{C}_{\rm kin}(\tau, \cdot) + \D \mathfrak{C}_{\rm int}(\tau, \cdot) \Big],\
E(0-) = E_0,
\end{equation}
for a.a. $\tau \in (0,T)$.

\end{Definition}

\begin{Remark} \label{Remnew}

Definition \ref{ED1} may seem rather awkward containing both the Young measure and the concentration defect measures. We show in Section
\ref{CR}, that the deviations of the Young measure from its barycenter can be included in the concentration defect. In particular,
we may always assume that $\mathcal{V}$ can be replaced by $\delta_{[\vr, \vm]}$ in \eqref{E9}--\eqref{E11}.
\end{Remark}

Dissipative measure--valued solutions enjoy an important property of weak--strong uniqueness. Here ``strong'' is meant in a generalized sense.
To state the relevant result, we introduce a class of functions $r$ and $\vu$:
\begin{equation} \label{E12}
\begin{split}
r \in C([0,T]; L^1(\Q)),\ \vu \in C([0,T] ; L^1(\Q; R^d));\\
0 < \underline r \leq r \leq \Ov{r},\ |\vu| \leq \Ov{\vu} \ \mbox{a.a. in}\ (0,T) \times \Omega;\\
r \in B^{\alpha, \infty}_p ([\delta, T] \times \Q),\
\vu \in B^{\alpha, \infty}_p ([\delta, T] \times \Q ; R^q) \ \mbox{for any}\ 0 < \delta < T, \
\alpha > \frac{1}{2},\ p \geq \frac{4 \gamma}{\gamma - 1};\\
\intO{ \left[ - \xi \cdot \vu(\tau, \cdot) (\xi \cdot \Grad) \varphi + D(\tau) |\xi|^2 \varphi \right] } \geq 0 \ \mbox{for a.a.}\
\tau \in (0,T),\\
\mbox{for any}\ \xi \in R^d \ \mbox{and any}\ \varphi \in C^1(\Q),\ \varphi \geq 0, \ \mbox{where}\ D \in L^1(0,T).
\end{split}
\end{equation}

The relevant weak--strong uniqueness principle reads, see \cite[Theorem 2.1]{FeiGhoJan}:

\begin{Proposition} [Weak--strong uniqueness] \label{EP1}

Let $\widehat{\vr} = r$, $\widehat{\vm} = r \vu$ be a weak solution to the Euler system \eqref{E1}, \eqref{E2} (specifically
the integral identities \eqref{E5}, \eqref{E6} are satisfied), where $r$, $\vu$ belong to the class \eqref{E12}.
Let $\mathcal{V}$, $\mathfrak{C}_{\rm kin}$, $\mathfrak{C}_{\rm int}$, $\mathfrak{C}_{\rm conv}$ be a dissipative measure--valued solution of the same problem with the initial data $[\vr_0, \vm_0, E_0]$ such that
\[
\vr_0 = r(0, \cdot),\
\vm_0 = r \vu (0, \cdot),\ E_0 =  \intO{ \left[ \frac{1}{2} r |\vu|^2 + P(r) \right](0, \cdot)}.
\]

Then
\[
\mathfrak{C}_{\rm kin} = \mathfrak{C}_{\rm int} = \mathfrak{C}_{\rm conv} = 0 \ \mbox{in} \ (0,T) \times \Q
\]
and
\[
\mathcal{V} =  \delta_{[\widehat{\vr}, \widehat{\vm}]}  \ \mbox{a.a. in}\ (0,T) \times \Q.
\]

\end{Proposition}

As shown in \cite{FeiGhoJan}, there are ``genuine'' weak solutions satisfying \eqref{E12}, in particular
the 1D rarefaction waves emanating from (discontinuous) Riemann data.

\subsection{Dissipative solutions}
\label{DiSo}

Following \cite{BreFeiHof19} we finally introduce the concept of dissipative solution of the Euler system:

\begin{Definition} [Dissipative solution] \label{ED2}

We say that
\[
\varrho \in C_{{\rm weak}}([0,T]; L^\gamma(\Q)),\ \vm \in C_{{\rm weak}}([0,T]; L^\frac{2\gamma}{\gamma + 1}(\Q; R^D)),
E \in BV [0,T]
\]
is a \emph{dissipative solution} of the Euler system \eqref{E1}, \eqref{E2}, with the initial data \eqref{E3} and the initial energy
$E_0$, if there exists a dissipative measure--valued solution in the sense of Definition \ref{ED1}, with the total energy $E$, such that
\[
\vr(t,x) = \left< \mathcal{V}_{t,x}; \tvr \right>,\
\vm(t,x) = \left< \mathcal{V}_{t,x}; \tvm \right> \ \mbox{for a.a.}\ (t,x) \in (0,T) \times \Q.
\]

\end{Definition}

Finally, we reformulate the weak--strong uniqueness principle in terms of dissipative solutions.

\begin{Proposition} [Weak--strong uniqueness] \label{EP2}

Let $\widehat{\vr} = r$, $\widehat{\vm} = r \vu$ be a weak solution to the Euler system \eqref{E1}, \eqref{E2} (specifically
the integral identities \eqref{E5}, \eqref{E6} are satisfied), where $r$, $\vu$ belong to the class \eqref{E12}.
Let $[\vr, \vm, E]$ be a dissipative solution of the same problem with the initial data $[\vr_0, \vm_0, E_0]$ such that
\[
\vr_0 = r(0, \cdot),\
\vm_0 = r \vu (0, \cdot),\ E_0 =  \intO{ \left[ \frac{1}{2} r |\vu|^2 + P(r) \right](0, \cdot)}.
\]

Then
\[
\vr = \widehat{\vr},\ \vm = \widehat{\vm}\ \mbox{a.a. in} \ (0,T) \times \Q,
\]
and
\[
E(\tau) = \intO{ \left[ \frac{1}{2} \frac{ |\widehat{\vm} |^2 }{\widehat{\vr}} + P(\widehat{\vr}) \right](\tau, \cdot)}
\ \mbox{for a.a.}\ \tau \in (0,T).
\]

\end{Proposition}

\section{Finite volume scheme for the isentropic Euler system}
\label{N}

We illustrate the abstract theory applying the results to the numerical solutions resulting from a finite volume approximation of the isentropic Euler system.  Our strategy is inspired by the fundamental Lax equivalence theorem \cite{Lax} on the convergence of
{\em consistent} and {\em stable} numerical schemes:
\begin{enumerate}
\item {\bf Existence.} We first recall existence of the approximate numerical solutions $\vrh$, $\vmh$, with the associated
energy $\Eh$ for any discretization parameter $h > 0$.

\item {\bf Stability.} We make sure that the scheme is \emph{energy dissipative}. In particular, we recover  the same energy bounds as for the continuous problem, including a discrete form of the energy inequality.

\item {\bf Consistency.} We establish a consistency formulation and find suitable bounds on the error terms.

\item {\bf Convergence.} Using the technique developed in Section \ref{A}, we associate to each sequence of numerical approximations
$\{ \vrhn, \vmhn \}_{n=1}^\infty$ its Young measure representation $\delta_{[\vrhn, \vmhn]}$. We perform the limit
$h_n \searrow 0$ and show that, up to a subsequence, the $\K-$limit of $\{ \delta_{[\vrhn, \vmhn]} \}_{n=1}^\infty$ is a Young measure
$\mathcal{V}$ associated to a dissipative measure--valued solution. In particular, we recover the strong convergence
of the arithmetic means
\[
\frac{1}{N} \sum_{n=1}^N \vrhn (t,x) \to \vr(t,x),\
\frac{1}{N} \sum_{n=1}^N \vmhn (t,x) \to \vm(t,x)\ \mbox{as}\ N \to \infty \ \mbox{for a.a.} \ (t,x) \in (0,T) \times \Q,
\]
where $\vr$, $\vm$ is a dissipative solution of the Euler system in the sense of Definition \ref{ED2}.

\item {\bf Unconditional convergence.} Applying the weak--strong uniqueness principle we show \emph{unconditional} strong
$L^1-$convergence of the numerical solutions provided that the Euler system admits a weak solution belonging to the regularity
class \eqref{E12}.

\end{enumerate}

We infer that the Young measure framework can substitute the linearity property. Accordingly, our result can be seen as a {\em nonlinear version of the Lax equivalence theorem}.

\subsection{Preliminaries  of  finite volume methods}

We start by introducing the basic notation concerning the mesh, temporal and spatial discretization, and discrete differential operators.
We recall that the spatial domain coincides with the flat torus $\Q$.  We shall write  $A \aleq B$ if $A \leq cB$ for a generic positive constant $c$ independent of $h.$

\subsubsection{Mesh}

The  grid $\mathcal{T}$ is a family of compact  parallelepiped elements,
\[
\Q =  \bigcup_{K \in \mathcal{T}} K, \ K = \prod_{i=1}^d [0, h_i] + x_K, \ 0 < \lambda h \leq h_i \leq h,\
i=1,\dots,d, \ 0< \lambda < 1,
\]
where $h$ denotes the mesh size and $x_K$ the position of the center of mass of an element $K.$
The intersection
$K \cap L$ of two elements $K,L \in \mathcal{T}$, $K \ne L$, is either empty, or a common vertex, a common edge, or a common
face.

The symbol $\mathcal{E}$ denotes the set of all faces $\sigma$.  To each face we associate a normal vector $\vn$. We write $\mathcal{E}(K)$ for the family of all boundary faces of an element $K$, $K | L = \mathcal{E}(K) \cap \mathcal{E}(L)$.

\subsubsection{Discrete function spaces}

We denote $\mathcal{Q}_h$ the space of $L^\infty$ functions constant on each element $K \in \mathcal{T}$, with the associated
projection:
\[
\Pi_{\mathcal{T}}: L^1(\Q) \to \mathcal{Q}_h,\ \Pi_{\mathcal{T}} v = \sum_{K \in \mathcal{T}} 1_K \frac{1}{|K|} \int_K
v \dx.
\]

\subsubsection{Numerical flux}

For a piecewise continuous function $v$ we define
\[
v^{\rm out}(x) = \lim_{\delta \to 0+} v(x + \delta \vc{n}),\
v^{\rm in}(x) = \lim_{\delta \to 0+} v(x - \delta \vc{n}),\
\Ov{v}(x) = \frac{v^{\rm in}(x) + v^{\rm out}(x) }{2},\
\jump{ v(x) }  = v^{\rm out}(x) - v^{\rm in}(x)
\]
whenever $x \in \sigma \in \mathcal{E}$.

Given a velocity field $\vc{v} \in \mathcal{Q}_h(\Q; R^d)$ and a transported scalar function $r \in \mathcal{Q}_h$, the
\emph{upwind flux} is defined as
\[
{\rm Up}[r, \vc{v}] = r^{\rm up} \vc{v} \cdot \vc{n} = r^{\rm in} [\vc{v} \cdot \vn]^+ + r^{\rm out} [\vc{v} \cdot \vc{n} ]^-
= \Ov{r} \ \Ov{\vc{v}} \cdot \vn - \frac{1}{2} |\Ov{\vc{v}} \cdot \vc{n} | \jump{r}
\]
on any face $\sigma$,
where
\[
[f]^{\pm} = \frac{f\pm |f|}{2} \quad \mbox{and} \quad
r^{\rm up} =
\begin{cases}
 r^{\rm in} & \mbox{if} \ \Ov{\vc{v}} \cdot \vc{n} \geq 0, \\
r^{\rm out} & \mbox{if} \  \Ov{\vc{v}} \cdot \vc{n} < 0.
\end{cases}
\]

Finally, we define the \emph{numerical flux},
\begin{align}\label{num_flux}
F_h(r,\vv)
={\rm Up}[r, \vv] - h^{\alpha} \jump{ r } =
 \Ov{r} \ \Ov{\vc{v}} \cdot \vn - \frac{1}{2}\Big( h^\alpha + |\Ov{\vc{v}} \cdot \vc{n} | \Big) \jump{r}, \quad \alpha >0.
\end{align}
If $\vc{r}$ is a vector the numerical flux is defined componentwise.

\subsubsection{Time discretization}
For a given time step $\TS \approx h>0$,
we denote the approximation of a function  $v$ at time $t^k= k\TS$ by $v^k$ for $k=1,\ldots,N_T(=T/\TS)$.  The time derivative is discretized by the backward Euler method,
\[
D_t  v^k = \frac{ v^k-  v^{k-1}}{\TS},\ \mbox{ for } k=1,2,\ldots, N_T.
\]
We introduce the piecewise constant functions on time interval,
\begin{equation*}
v(t,\cdot)  = v^0 \mbox{ for }  t < \TS,\ v(t,\cdot)= v^k \mbox{ for } t\in [k\TS,(k+1)\TS),\ k=1,2,\ldots,N_T.\\
\end{equation*}
Finally, we set
\[
 D_t  v = \frac{ v(t,\cdot) -  v(t - \Delta t,\cdot)}{\TS}
 \quad \mbox{ for } t \in [0,T].
\]

\subsection{Numerical scheme}
Using the above notation we are ready to introduce an implicit  finite volume scheme to approximate the Euler system \eqref{E1}.

\medskip

\noindent
Given the initial values  $(\vrh^0,\vuh^0) =(\Pim\vr_0, \Pim\vu_0),$ find $(\vr_h^k,\vu_h^k) \in \mathcal{Q}_h\times \mathcal{Q}_h$ satisfying for $k=1,\ldots,N_T$ the following equations
\begin{subequations}\label{scheme}
\begin{align}
&\intO{ D_t \vrh^k \varphi_h } - \sum_{ \sigma \in \facesint } \intSh{  F_h(\vrh^k,\vuh^k)
\jump{\varphi_h}   } = 0 \quad \mbox{for all } \varphi_h \in \mathcal{Q}_h,\label{scheme_den}\\
&\intO{ D_t  (\vrh^k \vuh^k) \cdot \bfphi_h } - \sum_{ \sigma \in \facesint } \intSh{ {\bf F}_h(\vrh^k \vuh^k,\vuh^k)
\cdot \jump{\bfphi_h}   }- \sum_{ \sigma \in \facesint } \intSh{ \Ov{p(\vr_h^k)} \vc{n} \cdot \jump{ \bfphi_h }  } \nonumber \\
&= - h^\beta \sum_{ \sigma \in \facesint } \intSh{ \jump{ \vuh^k }  \cdot \jump{ \bfphi_h }  },
\quad \mbox{for all }
\bfphi_h \in \mathcal{Q}_h( \Q; R^d) , \  \beta > -1. \label{scheme_mom}
\end{align}
\end{subequations}
The weak formulation \eqref{scheme}  of the scheme can be rewritten in the standard  per cell finite volume formulation
for all $K \in \grid$:
\begin{equation}\label{scheme_fv}
\begin{aligned}
& (\vrh^0,\vuh^0) =(\Pim\vr_0, \Pim\vu_0), \\
&D_t \vr^k _K + \sum_{\sigma \in \facesK} \frac{|\sigma|}{|K|} F_h(\vrh^k,\vuh^k) =0,
 \\
&D_t (\vrh^k \vuh^k)_K + \sum_{\sigma \in \facesK} \frac{|\sigma|}{|K|}
\left({\bf F}_h(\vrh^k \vuh^k,\vuh^k)  + \Ov{p(\rho_h^k)} \vc{n}
- h^{\beta} \jump{\vuh^k}\right) =0.
\end{aligned}
\end{equation}

Similarly to \cite{FeiLukMizShe}, we prefer the formulation in primitive variables $(\vr, \vu)$ instead of  conservative ones $(\vr, \vm
= \vr \vu)$. Indeed the scheme (\ref{scheme_fv}) mimics the  physical process of \emph{vanishing
viscosity limit} in the Navier--Stokes system. As a result
we get uniform stability estimates on
$\vrh$ and $\vu_h$ without any CFL condition.

\subsection{Solvability of the numerical scheme}

Here and hereafter, we impose a technical but physically grounded hypothesis
\[
1 < \gamma < 2,
\]
noting the physical range of the adiabatic exponent for gases $1 < \gamma \leq \frac{5}{3}$.

At the discrete level, the scheme \eqref{scheme_fv} coincides with that one used for discretization of the Navier--Stokes system in \cite{FeiLukMizShe}, with the viscosity  coefficients $\mu = - \lambda \equiv h^{\beta + 1}$. As shown by Ho\v sek and She \cite{HosShe}, there exists a numerical solution
$[\vrh^k, \vuh^k]$, or extended in time as $[\vrh, \vuh]$, for any given initial data and any $h > 0$. Moreover, we have
\[
\vrh^k > 0 \ \mbox{for any fixed}\ h > 0, \ k = 1,\dots \ \mbox{whenever}\ \vrh^0 > 0.
\]

\subsubsection{Discrete energy balance}

The scheme \eqref{scheme} is energy dissipative. More specifically, as shown in \cite[Theorem 3.3]{FeiLukMizShe},
\begin{equation} \label{N1}
\begin{split}
D_t &\intO{ \left[ \frac{1}{2} \vr^k_h |\vuh^k|^2 + P(\vrh^k) \right] }
+ \sum_{\sigma \in \mathcal{E}} \int_\sigma \left( h^\alpha \Ov{\vrh^k} \jump{ \vuh^k }^2 + h^\beta \jump{ \vuh^k }^2  \right)
\ds \\
&= - \frac{\Delta t}{2} \intO{ P''(\xi) |D_t \vrh^k |^2 } - \frac{1}{2} \sum_{\sigma \in \mathcal{E}} \int_\sigma
P''(\eta) \jump{ \vrh^k }^2 \left( h^\alpha + |\Ov{\vuh^k} \cdot \vc{n}| \right) \ds\\
&- \frac{\Delta t}{2} \intO{ \vrh^{k-1} |D_t \vuh^k|^2 }  - \frac{1}{2}\sum_{\sigma \in \mathcal{E}} \int_\sigma
(\vrh^k)^{\rm up} |\Ov{\vuh^k} \cdot \vc{n}| \jump{ \vuh^k }^2 \ds,
\end{split}
\end{equation}
where $\xi \in {\rm co}\{ \vrh^{k-1}, \vrh^k \}$, $\eta \in {\rm co}\{ \vr^k_K, \vr^k_L \}$ with the notation $\co{A}{B} \equiv [ \min\{A,B\} , \max\{A,B\}].$

\subsubsection{Consistency formulation}

In the consistency formulation, we rewrite the scheme in the form of the limit system perturbed by consistency errors. Referring to
\cite[Theorem 4.1]{FeiLukMizShe}, we have:
\begin{equation*}
- \intO{ \vrh^0 \varphi(0, \cdot) } = \int_0^T \intO{ \Big[ \vrh \partial_t \varphi + \vrh \vuh \cdot \Grad \varphi \Big] } \dt
+ \int_0^T \intO{ e_1(t, h, \varphi) } \dt
\end{equation*}
for any $\varphi \in C^3_c([0,T) \times \Q)$,
\begin{equation*}
\begin{split}
- \intO{ \vrh^0 \vuh^0 \cdot \bfphi(0, \cdot) } &= \int_0^T \intO{ \Big[ \vrh \vuh \cdot \partial_t \bfphi + \vrh \vuh \otimes
\vuh : \Grad \bfphi + p(\vrh) \Div \bfphi \Big] } \dt\\
&- h^\beta \int_0^T \sum_{\sigma \in \mathcal{E}} \int_\sigma \jump{ \vuh^k } \cdot \jump{ \Pim \bfphi }\ \ds \dt
+ \int_0^T \intO{ e_2(t, h, \bfphi) } \dt
\end{split}
\end{equation*}
for any $\bfphi \in C^3_c([0,T) \times \Q; R^d)$.
The error terms $e_1(t, h, \varphi)$, $e_2(t, h, \bfphi)$ were identified in \cite[Section~4]{FeiLukMizShe}.
Our goal is to show that these  error terms vanish in the asymptotic limit $h \to 0$. Similarly to \cite{FeiLukMizShe}, the necessary bounds are deduced
from the energy inequality \eqref{N1}. We focus only on the integrals depending on the velocity that must be handled
differently in the present setting. These are:
\begin{equation*}
\begin{aligned}
& E_1(r_h)= \frac12 \int_0^T \sum_{\sigma \in \facesint} \intSh{
  |\Ov{\vuh} \cdot \vc{n}| \jump{r_h }  \jump{ \Pim \varphi}  }  \dt   ,
  \\& E_2(r_h)= \frac14  \int_0^T \sum_{\sigma \in \facesint} \intSh{
 \jump{\vuh} \cdot \vc{n}   \jump{r_h }  \jump{  \Pim \varphi}  }  \dt
\end{aligned}
\end{equation*}
for $r_h$ being $\vr_h$ or $\vr_h u_{i,h},$ $i=1,\ldots,d.$

 Analogously to \cite{FeiLukMizShe}, we get
\begin{align*}
 E_1(r_h) & \aleq  h || \varphi ||_{C^3} ||r_h ||_{L^2(0,T;L^2(\Omega))}  \left[ \|\vu_h\|_{L^2(0,T;L^2(\Omega))} +  \left(  \int_0^T \sum_{\sigma \in \facesint}\intSh{  \frac{\jump{\vu_h }^2 }{h}} \right)^{1/2}
\right]
\\ &\aleq  h || \varphi ||_{C^3}  ||r_h ||_{L^2(0,T;L^2(\Omega))}  \left[ 1+
\left( \int_0^T \sum_{\sigma \in \facesint} \intSh{  \frac{\jump{\vu_h }^2 }{h}} \right)^{1/2}
 \right] \\
& \aleq h   || \varphi ||_{C^3}\ h^{-\frac{\alpha+2}{2\gamma}}  \left[ 1 + h^\frac{-(\beta + 1)}{2}\right]
 \aleq h^{\delta_1} || \varphi ||_{C^3} \quad \mbox { for }\delta_1 = 1 - \left(\frac{\alpha + 2 }{2 \gamma} + \frac{\beta + 1}{2}\right),
\end{align*}
where \cite[Lemma~2.4]{FeiLukMizShe} combined with  \eqref{N1} and the  estimates from \cite[Lemma~3.5]{FeiLukMizShe} have been used to control the norms $||\vu_h ||_{L^2(0,T;L^2(\Omega))}  $ and
$||r_h ||_{L^2(0,T;L^2(\Omega))},$ respectively.

Furthermore, we have
\begin{equation*}
\begin{aligned}
E_2(\vr_h) &  \aleq
 h ||\varphi ||_{C^3} \left( \int_0^T \sum_{\sigma \in \facesint} \intSh{ \jump{\vuh}^2 }  \dt  \right)^{1/2}  \left( \int_0^T \sum_{\sigma \in \facesint} \intSh{ \jump{\vrh}^2 }  \dt  \right)^{1/2}
\\& \aleq h ||\varphi ||_{C^3} \left( \int_0^T \sum_{\sigma \in \facesint} \intSh{ \jump{\vuh}^2 }  \dt  \right)^{1/2}
\left( \int_0^T \sum_{\sigma \in \facesint} \intSh{ \Ov{\vrh}^2 }  \dt  \right)^{1/2}
\\&
\aleq h^{\frac 1 2 -\frac{\beta}{2}} ||\varphi ||_{C^3} ||\vrh||_{L^2L^2}
\aleq h^{\delta_1} ||\varphi ||_{C^3}, \quad \mbox { for }\delta_1 = 1 - \left(\frac{\alpha + 2 }{2 \gamma} + \frac{\beta + 1}{2}\right),
\end{aligned}
\end{equation*}
and exactly as in \cite{FeiLukMizShe},
\begin{align*}
 E_2(\vr_h\vu_h) &\aleq  h \left( \int_0^T \sum_{\sigma \in \facesint} \intSh{  \Ov{\vrh } \jump{\vuh}^2 } \right)^{1/2}
 \left( \int_0^T \sum_{\sigma \in \facesint} \intSh{  \Ov{\vrh } |\Ov{\vuh}|^2 } \right)^{1/2} \\
 &+ h \int_0^T \sum_{\sigma \in \facesint} \intSh{  \Ov{\vrh } \jump{\vuh}^2 } \aleq h^{(1-\alpha)/2}+h^{1-\alpha} \aleq h^{\delta_2} \quad \mbox { for }\delta_2=\frac{1-\alpha}{2}.
\end{align*}
Clearly, $\delta_1,$ $\delta_2 > 0$ whenever $-1 < \beta < 1- \frac{\alpha + 2}{\gamma}$ and $\alpha<1.$

Finally, diffusive correction in the momentum equation
$$d(h,\bfphi):= -h^\beta \int_0^T \sum_{\sigma \in \mathcal{E}} \int_\sigma \jump{ \vuh^k } \cdot
\jump{ \Pim \bfphi }\ \ds \dt
$$
 can be handled as follows
\begin{align*}
&\left|h^\beta \int_0^T \sum_{\sigma \in \mathcal{E}} \int_\sigma \jump{ \vuh^k } \cdot
\jump{ \Pim \bfphi }\ \ds \dt \right|
\leq h^\frac{\beta + 1}{2}\left(  h^\beta\ \int_0^T \sum_{\sigma \in \mathcal{E}} \int_\sigma \jump{ \vuh^k } ^2 \ds \dt \right)^{1/2} \cdot
\\
&\left(  \int_0^T \sum_{\sigma \in \mathcal{E}} \int_\sigma \frac{\jump{ \Pim \bfphi }^2}{h} \ds \dt \right)^{1/2}
\aleq h^{\frac{\beta+1}{2}} || \bfphi  ||_{C^3} \to 0 \ \mbox { as } h \to 0 \ \mbox{ for } \beta >-1.
\end{align*}

Combining the above estimates with those obtained in \cite[Section 4]{FeiLukMizShe}, we obtain:

\begin{Proposition} \label{NP1}

Let
\begin{equation} \label{Npar}
 0<\alpha <1 ,\
-1 < \beta < \left( 1 - \frac{2}{\gamma} \right) - \frac{\alpha}{\gamma}.
\end{equation}
Then the finite volume method \eqref{scheme_fv} is consistent with the Euler  equations \eqref{E1}, i.e.
\begin{equation} \label{N2}
- \intO{ \vrh^0 \varphi(0, \cdot) } = \int_0^T \intO{ \Big[ \vrh \partial_t \varphi + \vrh \vuh \cdot \Grad \varphi \Big] } \dt
+ \int_0^T \intO{ e_1(t, h, \varphi) } \dt
\end{equation}
for any $\varphi \in C^3_c([0,T) \times \Q)$,
\begin{equation} \label{N3}
\begin{split}
- \intO{ \vrh^0 \vuh^0 \cdot \bfphi(0, \cdot) } &= \int_0^T \intO{ \Big[ \vrh \vuh \cdot \partial_t \bfphi + \vrh \vuh \otimes
\vuh : \Grad \bfphi + p(\vrh) \Div \bfphi \Big] } \dt\\
&+ d(h, \bfphi) + \int_0^T \intO{  e_2(t, h, \bfphi) } \dt
\end{split}
\end{equation}
for any $\bfphi \in C^3_c([0,T) \times \Q; R^d)$, with the consistency errors
\[
\begin{split}
&\left|\int_0^T \intO{ e_1(t,h, \varphi)  } \dt \right| \aleq h^\delta \| \varphi \|_{C^3},\\
&\left|\int_0^T \intO{ e_2(t,h, \bfphi) } \dt \right| \aleq h^\delta \| \bfphi \|_{C^3}, \\
& \Big | d(h, \bfphi) \Big | \aleq h^\frac{\beta +1}{2} \| \bfphi \|_{C^3}
\end{split}
\]
for a certain $\delta > 0$.
\end{Proposition}

\subsection{Convergence}

At this stage, we are ready to apply the machinery developed in Section \ref{A}. Recall that
\[
Q = (0,T) \times \Q \subset R^{d + 1},\ S = [0, \infty) \times R^d \subset R^{d+1}.
\]
For a sequence of discretization parameters $h_n \searrow 0$, consider a sequence
\[
\vU_{h_n} = [ \vr_{h_n}, \vm_{h_n} ], \ \vm_{h_n} = \vr_{h_n} \vu_{h_n}
\]
of numerical solutions resulting from \eqref{scheme_fv}. In view of the energy bounds \eqref{N1}, there exists a subsequence
$n_k \to \infty$ such that for $\vU_k \equiv [ \vr_k, \vm_k ] \equiv [\vr_{h_{n_k}}, \vm_{h_{n_k}}]$ we get:
\begin{equation} \label{N4}
\begin{split}
\frac{1}{2} \frac{|\vm_k|^2}{\vr_k} &\to \mathfrak{M}_{\rm kin} \ \mbox{weakly-(*) in}\
L^\infty(0,T; \mathcal{M}^+(\Q)),\\
P(\vr_k) &\to \mathfrak{M}_{\rm int} \ \mbox{weakly-(*) in}\
L^\infty(0,T; \mathcal{M}^+(\Q)), \\
\frac{\vm_k \otimes \vm_k}{\vr_k} &\to \mathfrak{M}_{\rm conv} \ \mbox{weakly-(*) in}\
L^\infty(0,T; \mathcal{M}(\Q; R^{d \times d}_{\rm sym})),\\
p(\vr_k) &\to (\gamma - 1) \mathfrak{M}_{\rm int} \ \mbox{weakly-(*) in}\
L^\infty(0,T; \mathcal{M}^+(\Q))
\end{split}
\end{equation}
as $k \to \infty$.

Next, by the same token, we observe that $[\vr_k, \vm_k]$ possesses the weak$-\K$ property. In particular, passing to another
subsequence as the case may be, we may suppose that
\begin{equation} \label{N5}
\delta_{[\vr_k, \vm_k]} \toKY \mathcal{V},
\end{equation}
where $\mathcal{V}_{t,x} \in \mathcal{P}(S)$, $(t,x) \in Q$ is a Young measure. Moreover, by virtue of Proposition \ref{PP1},
\begin{equation} \label{N6}
\delta_{[\vr_k, \vm_k]} \toNY \mathcal{V}.
\end{equation}
We denote
\[
\vr(t,x) = \left< \mathcal{V}_{t,x}; \tvr \right>,\ \vm(t,x) = \left< \mathcal{V}_{t,x}; \tvm \right>,\
(t,x) \in Q.
\]

Evoking again the energy bound \eqref{N1}, we deduce that the convex functions
\[
[\tvr, \tvm] \mapsto \frac{|\tvm|^2}{\tvr},\ \tvr \mapsto P(\tvr)
\]
are $\mathcal{V}_{t,x}$ integrable for a.a. $(t,x) \in Q$. We set
\begin{equation} \label{N7}
\begin{split}
\mathfrak{C}_{\rm kin} &= \mathfrak{M}_{\rm kin} - \left<\mathcal{V}_{t,x}; \frac{1}{2} \frac{|\tvm|^2}{\tvr} \right> (dx \otimes \dt),\\
\mathfrak{C}_{\rm int} &= \mathfrak{M}_{\rm int} - \left<\mathcal{V}_{t,x}; P(\tvr) \right> (dx \otimes \dt),\\
\mathfrak{C}_{\rm conv} &= \mathfrak{M}_{\rm conv} - \left<\mathcal{V}_{t,x}; \frac{ \tvm \otimes \tvm }{\tvr} \right> (dx \otimes \dt).
\end{split}
\end{equation}

We show that the concentration defect satisfies the compatibility condition \eqref{E8}. Let
\[
\chi_\ep (Y) = Y \ \mbox{for}\ 0 \leq Y \leq \frac{1}{\ep},\ \chi_\ep(Y) = \frac{1}{\ep} \ \mbox{for}\
Y \geq \frac{1}{\ep} .
\]
For $\mathbb{M} \in C([0,T] \times \Q; R^{d \times d}_{\rm sym})$, we have
\[
\begin{split}
\int_0^T &\int_{\Q} \mathbb{M} : \D \mathfrak{M}_{\rm conv}(t) \dt =
\lim_{k \to \infty} \int_0^T \intO{ \mathbb{M} : \frac{ \vm_k \otimes \vm_k }{\vr_k} } \dt \\ &=
\lim_{k \to \infty} \int_0^T \intO{ \left[ \mathbb{M} : \frac{ \vm_k }{|\vm_k|}  \otimes \frac{\vm_k}{|\vm_k| }
\frac{ \chi_\ep (|\vm_k|^2)} {\vr_k + \ep}  + \mathbb{M} : \frac{ \vm_k }{|\vm_k|}  \otimes \frac{\vm_k}{|\vm_k| }
\left(\frac{|\vm_k|^2}{\vr_k} - \frac{ \chi_\ep (|\vm_k|^2)} {\vr_k + \ep} \right) \right]} \dt\\
&= \int_0^T \intO{ \mathbb{M}: \left< \mathcal{V}_{t,x};  \frac{ \tvm }{|\tvm|}  \otimes \frac{\tvm}{|\tvm| }
\frac{ \chi_\ep (|\tvm|^2)} {\tvr_k + \ep}  \right> } \dt \\
&+ \lim_{k \to \infty} \int_0^T \intO{  \mathbb{M} : \frac{ \vm_k }{|\vm_k|}  \otimes \frac{\vm_k}{|\vm_k| }
\left(\frac{|\vm_k|^2}{\vr_k} - \frac{ \chi_\ep (|\vm_k|^2)} {\vr_k + \ep} \right) } \dt
\end{split}
\]
Letting $\ep \to 0$ we may use the Lebesgue convergence theorem to conclude that
\[
\int_0^T \int_{\Q} \mathbb{M} : \D \mathfrak{C}_{\rm conv}(t) \dt =
\lim_{\ep \to 0} \left[ \lim_{k \to \infty} \int_0^T \intO{  \mathbb{M} : \frac{ \vm_k }{|\vm_k|}  \otimes \frac{\vm_k}{|\vm_k| }
\left(\frac{|\vm_k|^2}{\vr_k} - \frac{ \chi_\ep (|\vm_k|^2)} {\vr_k + \ep} \right) } \dt \right]
\]
which implies \eqref{E8}.

Finally, as the discrete energy is non--increasing, we may apply Helly's selection theorem obtaining
\[
E_k \equiv \intO{ \left[ \frac{1}{2} \vr_k |\vu_k|^2 + P(\vr_k) \right] } \to E \ \mbox{pointwise in}\ [0,T],
\]
where $E$ is a non--decreasing function in $[0,T]$, with
\[
E(0-) = \intO{ \left[ \frac{1}{2} \frac{|\vm_0|^2}{\vr_0} + P(\vr_0) \right] }.
\]

Combining the previous observations with the estimates of the consistency errors established in Proposition \ref{NP1}, we obtain
the main result concerning convergence of the numerical scheme \eqref{scheme_fv}.

\begin{Theorem} \label{NT1}
Let the initial data $\vr_0$, $\vm_0$ belong to the class
\[
\vr_0 \in L^\gamma(\Q), \ \vr_0 > 0, \ E_0 = \intO{ \left[ \frac{1}{2} \frac{|\vm_0|^2}{\vr_0} + P(\vr_0) \right] } < \infty.
\]
Let
\[
0 <\alpha <1 , \
-1 < \beta < \left( 1 - \frac{2}{\gamma} \right) - \frac{\alpha}{\gamma}.
\]
Finally, let $[\vr_{h_n}, \vm_{h_n} = \vr_{h_n} \vu_{h_n}]$ be a sequence of numerical solutions resulting from the
scheme \eqref{scheme_fv} with $h = h_n \searrow 0$.

Then there is a subsequence $n_k \to \infty$ such that:
\begin{itemize}
\item
\[
\delta_{ [\vr_{h_{n_k}}, \vm_{h_{n_k}} ] } \toKY \mathcal{V} \ \mbox{as}\ k \to \infty,
\]
where $\mathcal{V}$ is a Young measure on $(0,T) \times \Q$;
\item $\vr_k = \vr_{h_{n_k}}$, $\vm_k = \vm_{h_{n_k}}$ generate the measures $\mathfrak{M}_{\rm kin}$, $\mathfrak{M}_{\rm int}$,
$\mathfrak{M}_{\rm conv}$ via the limits in \eqref{N4};
\item the Young measure $\mathcal{V}$, together with the concentrations measures defined through \eqref{N7},
represent a dissipative measure--valued solution of the Euler system, with the initial data $\vr_0$, $\vm_0$, and the initial
energy $E(0-) = E_0$.
\end{itemize}

In particular,
\[
\begin{split}
\vr_{h_{n_k}} &\to \vr \ \mbox{weakly-(*) in}\ L^\infty(0,T; L^\gamma(\Q)), \\
\vm_{h_{n_k}} &\to \vm \ \mbox{weakly-(*) in}\ L^\infty(0,T; L^{\frac{2 \gamma}{\gamma + 1}}(\Q; R^d)),
\end{split}
\]
and
\begin{equation} \label{N8}
\begin{split}
E_{h_{n_k}} &\to E \ \mbox{pointwise in}\ [0,T],\\
\frac{1}{N} \sum_{k = 1}^N \vr_{h_{n_k}} &\to \vr \ \mbox{as}\ N \to \infty \ \mbox{a.a. in}\ (0,T) \times \Q,\\
\frac{1}{N} \sum_{k = 1}^N \vm_{h_{n_k}} &\to \vm \ \mbox{as}\ N \to \infty \ \mbox{a.a. in}\ (0,T) \times \Q,
\end{split}
\end{equation}
where $[\vr, \vm, E]$ is a dissipative solution of the Euler system with the initial data $\vr_0$, $\vm_0$, $E_0$.

\end{Theorem}

\begin{Remark} \label{NR1}

Theorem~\ref{NT1} holds even for $\vr \geq 0$.  In that case, the initial data for the approximate density must be taken
\[
\vr^0_h =
\Pim \vr_0 + h.
\]

\end{Remark}

The main novelty with respect to the existing results is the strong convergence of the arithmetic averages of the numerical solutions
established in \eqref{N8}. In view of the energy bounds, relation (\ref{N8}) implies
\[
\begin{split}
\frac{1}{N} \sum_{k = 1}^N \vr_{h_{n_k}} &\to \vr \ \mbox{as}\ N \to \infty \ \mbox{in}\ L^1((0,T) \times \Q),\\
\frac{1}{N} \sum_{k = 1}^N \vm_{h_{n_k}} &\to \vm \ \mbox{as}\ N \to \infty \ \mbox{in}\ L^1((0,T) \times \Q; R^d).
\end{split}
\]

\subsubsection{Unconditional convergence}

Theorem \ref{NT1} asserts convergence of the numerical scheme up to a subsequence. Unconditional convergence holds provided the
continuous Euler system admits a unique (dissipative) solution. According to the weak--strong uniqueness principle, this is the case
provided $\vr$ and $\vu$ belong to the regularity class \eqref{E12}. Combining Theorem \ref{NT1}, Proposition \ref{EP1}, and
Lemma \ref{SL1}, we obtain the following result.

\begin{Theorem} \label{NT2}
Let the initial data $\vr_0$, $\vm_0$ belong to the class
\[
\vr_0 \in L^\gamma(\Q), \ \vr_0 > 0, \ E_0 = \intO{ \left[ \frac{1}{2} \frac{|\vm_0|^2}{\vr_0} + P(\vr_0) \right] } < \infty.
\]
Suppose that the Euler system \eqref{E1}, \eqref{E2} admits a weak solution $\vr$, $\vm = \vr \vu$, where $\vr$ and $\vu$ belong to
the class \eqref{E12}.
Let $[\vr_{h_n}, \vm_{h_n} = \vr_{h_n} \vu_{h_n}]$ be a sequence of numerical solutions resulting from the
scheme \eqref{scheme_fv} with $h = h_n \searrow 0$, and
\[
0<\alpha <1,\
-1 < \beta < \left( 1 - \frac{2}{\gamma} \right) - \frac{\alpha}{\gamma}.
\]

Then
\[
\vr_{h_n} \to \vr\ \mbox{in}\ L^1((0,T) \times \Q), \ \vm_{h_n} \to \vm \ \mbox{in}\ L^1((0,T) \times \Q; R^d) \ \mbox{as}\ n \to \infty.
\]

\end{Theorem}

Note that Lipschitz (strong) solutions of the Euler system with strictly positive density obviously belong to the class \eqref{E12}; whence Theorem \ref{NT2} yields unconditional strong convergence to the strong solution as long as the latter exists.

It is known, see e.g. Benzoni--Gavage and Serre \cite{BenSer}, Majda \cite{Majd}, that the Euler system admits local--in--time strong solutions
for sufficiently regular initial data, specifically,
\[
\vr_0 > 0, \ \vr_0 \in W^{k,2}(\Q),\ \vu_0 \in W^{k,2}(\Q; R^d),\ k > \left[ \frac{d}{2} \right] + 1.
\]
The strong solution exists on a maximal time interval  $[0,T_{\rm max})$ and stay regular as long as its derivatives remain uniformly bounded,
\[
\limsup_{t \to T_{\rm max}-} \| \Grad \vr \|_{L^\infty(\Q; R^d)} \to \infty \ \mbox{and/or} \
\limsup_{t \to T_{\rm max}-} \| \Grad \vu \|_{L^\infty(\Q; R^{d \times d})} \to \infty,
\]
 see, e.g., Alinhac \cite{Alinh}. Thus Theorem \ref{NT2} yields the following result.

\begin{Theorem} \label{NT3}
Let the initial data $\vr_0$, $\vm_0$ belong to the class
\[
\vr_0 > 0, \ \vr_0 \in W^{k,2}(\Q),\ \vu_0 \in W^{k,2}(\Q; R^d),\ k > \left[ \frac{d}{2} \right] + 1.
\]
Let $[\vr_{h_n}, \vm_{h_n} = \vr_{h_n} \vu_{h_n}]$ be a sequence of numerical solutions resulting from the
scheme \eqref{scheme_fv} with $h = h_n \searrow 0$, and
\[
0<\alpha <1,\
-1 < \beta < \left( 1 - \frac{2}{\gamma} \right) - \frac{\alpha}{\gamma}.
\]
In addition, suppose that
\begin{equation} \label{Lip}
\sup_{\sigma \in \mathcal{E}} \frac{ |\jump{\vr_{n_h}}| }{h} + \sup_{\sigma \in \mathcal{E}} \frac{ |\jump{\vu_{n_h}}| }{h} \leq L
\end{equation}
uniformly for $h_n \searrow 0$.

Then
\[
\vr_{h_n} \to \vr\ \mbox{in}\ L^1((0,T) \times \Q), \ \vm_{h_n} \to \vm \ \mbox{in}\ L^1((0,T) \times \Q; R^d) \ \mbox{as}\ n \to \infty
\]
where $\vr$, $\vu$ is a classical solution of the Euler system  \eqref{E1}, \eqref{E2}.

\end{Theorem}

Indeed Theorem \ref{NT2} guarantees the strong convergence to the unique classical solution in $(0, T_{\rm max}) \times \Q$. In view of hypothesis \eqref{Lip}, the limit is uniformly Lipschitz; whence $T_{\rm max} = T$.

\subsubsection{Absence of smooth solution}

We conclude the discussion by presenting a negative result concerning the absence of the strong solution should the numerical solutions develop oscillations. In accordance with
Proposition \ref{PPP}, the strong convergence of numerical solutions $[\vr_{h_n}, \vm_{h_n}]$ implies
\[
[\vr_{h_n}, \vm_{h_n}] \toSK [\vr, \vm].
\]
This observation yields the following result.

\begin{Theorem} \label{NT4}

Under the hypotheses of Theorem \ref{NT2} suppose that there exists a function $g \in BC(R^{d+1})$ such that
\begin{equation} \label{N10}
\frac{1}{N} \sum_{k=1}^N g(\vr_{h_{n_k}}, \vm_{h_{n_k}}) \to G \ne g(\vr, \vm) \ \mbox{on a set of positive measure in}\ (0,T) \times \Q.
\end{equation}

Then the Euler system \eqref{E1}, \eqref{E2} does not admit a classical solution in $(0,T) \times \Q$ for the initial data $\vr_0$, $\vm_0$.

\end{Theorem}

Note that the limit on the left--hand side of \eqref{N10} always exists as the numerical solutions $\K-$converge. However, the limit must be a Dirac mass $\delta_{[\vr, \vm]}$ applied to $g$ should the continuous solution
exist, in contrast with \eqref{N10}.

\section{Concluding remarks}
\label{CR}

We have extended the concept of $\K-$convergences to sequences of numerical solutions approximating the models of inviscid fluids in continuum fluid mechanics. We have also introduced a class of generalized solutions -- the dissipative solutions to the isentropic Euler system.

To illustrate the theoretical results, we have studied  a  finite volume method as an numerical approximation of the isentropic Euler system on a periodic spatial domain.  Note that a particular vanishing viscosity term
$$
d(h,\bfphi):= -h^\beta \int_0^T \sum_{\sigma \in \mathcal{E}} \int_\sigma \jump{ \vuh^k } \cdot
\jump{ \Pim \bfphi }\ \ds \D t
$$
allowed us to obtain suitable stability estimates on the discrete velocity $\vu_h$ and unconditional consistency of the scheme. Consequently,
we have shown that, up to a subsequence, the arithmetic averages of numerical solutions converge
pointwise a.a.~to a dissipative solution of the  Euler system. The convergence is unconditional provided the limit solution is unique.
We have also shown a simple criterion based on the oscillatory behavior of the numerical sequence that indicates absence of smooth solution to the limit system.

Finally, it is interesting to note that in the context of the isentropic Euler system considered in this paper,
the dissipative solutions discussed in Section \ref{DiSo} can be defined without making reference to the Young measure $\mathcal{V}$.
More specifically, we can include the ``oscillation'' defects
\[
\mathfrak{D}_{{\rm conv}} (t) \equiv \left< \mathcal{V}_{t,x}; \frac{ \tvm \otimes \tvm }{\tvr} \right> -
\frac{\vc{m_i} \otimes \vc{m}_j }{\vr}(t,x) ,\ \mathfrak{D}_{\rm int} \equiv
\left< \mathcal{V}_{t,x}; P(\tvr) \right> - P(\vr) (t,x)
\]
and
\[
\mathfrak{D}_{\rm kin}(t,x) \equiv \left< \mathcal{V}_{t,x}; \frac{1}{2} \frac{ |\tvm|^2 }{\tvr} \right>
-  \frac{1}{2} \frac{ |\vm|^2}{\vr}
\]
in $\mathfrak{C}_{\rm conv}$, $\mathfrak{C}_{\rm int}$,
and $\mathfrak{C}_{\rm kin}$, respectively. Indeed this is obvious for $\mathfrak{D}_{\rm int}$ as, by Jensen's
inequality and convexity of $P$,
\[
\mathfrak{D}_{\rm int} \geq 0.
\]
Moreover, again obviously,
\[
\left< \mathcal{V}_{t,x}; p(\tvr) \right> - p(\vr) (t,x) =
(\gamma - 1) \mathfrak{D}_{\rm int}(t,x) \ \mbox{for a.a.}\ (t,x) \in (0,T) \times \Q.
\]

Next we observe that $\mathfrak{D}_{\rm conv}$ is a symmetric matrix for a.a. $(t,x)$ with
\[
\frac{1}{2} {\rm trace} [\mathfrak{D}_{\rm conv}] = \left< \mathcal{V}_{t,x}; \frac{1}{2} \frac{ |\tvm|^2 }{\tvr} \right>
-  \frac{1}{2} \frac{ |\vm|^2}{\vr} = \mathfrak{D}_{\rm kin} \geq 0.
\]
Thus it remains to show that $\mathfrak{D}_{\rm conv}$ is positively definite. To this end, we write
\[
\mathfrak{D}_{\rm conv} : (\xi \otimes \xi) = \left< \mathcal{V}_{t,x}; \frac{ |\tvm \cdot \xi|^2 }{\tvr} \right> -
\frac{ |\vm \cdot \xi|^2 }{\vr} \geq 0,
\]
where we have used convexity of the function
\[
[\vr, \vm] \mapsto \frac{|\vm \cdot \xi|^2}{\vr}, \ \xi \in R^d.
\]
Consequently, without loss of generality, we may replace $\mathcal{V}$ in \eqref{E9}--\eqref{E11} by
$\delta_{[\vr, \vm]}$.

\bigskip \noindent
\textbf{Acknowledgement.} E. Feireisl and  H. Mizerov\' a would like to thank DFG TRR 146 Multiscale simulation methods for soft matter systems and the Institute of Mathematics, University Mainz for the hospitality.

\def\cprime{$'$} \def\ocirc#1{\ifmmode\setbox0=\hbox{$#1$}\dimen0=\ht0
  \advance\dimen0 by1pt\rlap{\hbox to\wd0{\hss\raise\dimen0
  \hbox{\hskip.2em$\scriptscriptstyle\circ$}\hss}}#1\else {\accent"17 #1}\fi}
\bibliographystyle{plain}

\end{document}